\documentclass[11pt,reqno]{amsart}
\usepackage{graphicx}
\usepackage{amsmath,amssymb,amsfonts,euscript,enumerate}
\usepackage{amsthm}
\usepackage{color,wrapfig}
\usepackage{verbatim}
\newtheorem{thm}{Theorem}[section]
\newtheorem{cor}[thm]{Corollary}
\newtheorem{lemma}[thm]{Lemma}
\newtheorem{prop}[thm]{Proposition}

\newcommand{\R}{{\mathbb{R}}}
\newcommand{\Z}{{\mathbb{Z}}}

\newcommand{\4}{\widetilde}

\newcommand{\La}{\triangle}

\def\ni{\noindent}

\begin{document}
\title{Large time behaviour of higher dimensional logarithmic diffusion 
equation}
\author[Kin Ming Hui]{Kin Ming Hui}
\address{Kin Ming Hui:
Institute of Mathematics, Academia Sinica,
Taipei, 10617, Taiwan, R.O.C.}
\email{kmhui@gate.sinica.edu.tw}
\author[Sunghoon Kim]{Sunghoon Kim}
\address{Sunghoon Kim:
Institute of Mathematics, Academia Sinica,
Taipei, 10617, Taiwan, R.O.C.}
\email{gauss79@math.sinica.edu.tw}
\date{Nov 23, 2011}

\maketitle
\begin{abstract}
Let $n\ge 3$ and $\psi_{\lambda_0}$ be the radially symmetric solution of $\Delta
\log\psi+2\beta\psi+\beta x\cdot\nabla\psi=0$ in $\R^n$, $\psi(0)
=\lambda_0$, for some constants $\lambda_0>0$, $\beta>0$. Suppose 
$u_0\ge 0$ satisfies $u_0-\psi_{\lambda_0}\in L^1(\R^n)$ and $u_0(x)\approx 
\frac{2(n-2)}{\beta}\frac{\log |x|}{|x|^2}$ as $|x|\to\infty$. We prove 
that the rescaled solution $\4{u}(x,t)=e^{2\beta t}u(e^{\beta t}x,t)$ of the 
maximal global solution $u$ of the equation $u_t=\Delta\log u$ in $\R^n
\times (0,\infty)$, $u(x,0)=u_0(x)$ in $\R^n$, converges uniformly on every 
compact subset of $\R^n$ and in $L^1(\R^n)$ to $\psi_{\lambda_0}$ as 
$t\to\infty$. Moreover $\|\4{u}(\cdot,t)-\psi_{\lambda_0}\|_{L^1(\R^n)}
\le e^{-(n-2)\beta t}\|u_0-\psi_{\lambda_0}\|_{L^1(\R^n)}$ for all $t\ge 0$.
\end{abstract}
\vskip 0.2truein


\setcounter{equation}{0}
\setcounter{section}{0}

\section{Introduction}
\setcounter{equation}{0}
\setcounter{thm}{0}

In this paper we will study the asymptotic large time behaviour of the
solution of the equation
\begin{equation}\label{log-eqn}
\begin{cases}
u_t=\La\log u, u>0,\quad\mbox{ in }\R^n\times(0,\infty),\\
u(x,0)=u_0(x) \qquad\quad\mbox{ in }\R^n
\end{cases}
\end{equation}
for $n\ge 3$.
When $n=1$, P.L.~Lions and G.~Toscani have proved that \eqref{log-eqn}  arises
as the diffusive limit for finite velocity Boltzmann kinetic models \cite{LT}
and T.~Kurtz \cite{K} has showed that  \eqref{log-eqn} arises as the limiting 
distribution 
of two gases moving against each other and obeying the Boltzmann equation.  
When $n=2$, the above equation arises in the study of Ricci flow on the 
complete $\R^2$ \cite{W1}, \cite{W2}. \eqref{log-eqn} also arises as the
singular limit \cite{ERV}, \cite{H2}, as $m\to 0$ of the following class 
of degenerate parabolic equation,
\begin{equation}\label{porous-eqn}
\begin{cases}
u_t=\Delta(u^m/m)\quad\mbox{ in }\R^n\times(0,T),\\
u(x,0)=u_0(x) \quad \mbox{ in }\R^n.
\end{cases}
\end{equation}
It is known that \eqref{porous-eqn} arises in many physical models. For 
example when $m=1/2$, \eqref{porous-eqn} arises in the study of the diffusion 
of impurities in silicon \cite{Ki}. When $m>1$, \eqref{porous-eqn} arises in 
the study of gases through porous media \cite{A}, \cite{P}. Interested reader
can read the book \cite{DK} by P.~Daskalopoulos and C.E.~Kenig for the recent 
results on \eqref{log-eqn} and \eqref{porous-eqn}.

Existence of infinitely many finite mass solutions of \eqref{log-eqn} for 
$n=2$ and $0\le u_0\in L^1(\R^2)\cap L^p(\R^2)$ for some $p>1$ is proved by 
P.~Daskalopoulos and M.A.~del Pino \cite{DP1} and K.M.~Hui \cite{H1}.
Global existence and uniqueness of solutions of \eqref{log-eqn} for $n=2$ 
is proved by P.~Daskalopoulos and M.A.~del Pino \cite{DP1} and S.Y.~Hsu 
\cite{Hs1}. Global existence of solution of \eqref{log-eqn} for $n\ge 3$ 
is proved by P.~Daskalopoulos, M.A.~del Pino, and K.M.~Hui in \cite{DP2},
\cite{H3}. Large time behaviour of solution of
\eqref{log-eqn} for $n=2$ is proved by S.Y.~Hsu in \cite{Hs3}, \cite{Hs4}.

Extinction profile of solutions of \eqref{porous-eqn} for $0<m<(n-2)/n$
and $n\ge 3$ is studied by P.~Daskalopoulos and N.~Sesum in \cite{DS1}.
Extinction profile of maximal solutions of \eqref{log-eqn} in 
$\R^n\times (0,T)$ for $n=2$ near the extinction time $T>0$ is studied by
P.~Daskalopoulos, M.A.~del Pino, N.~Sesum and K.M.~Hui \cite{DP3}, 
\cite{DS2}, \cite{H4}.
Extinction profile of maximal solutions of \eqref{log-eqn} in 
$\R^n\times (0,T)$ for $n=3$ and $n\ge 5$ near the extinction time $T>0$ 
with initial value $u_0$ satisfying the condition
\begin{equation*}\label{u0-trapped-barenblat}
B_{k_1}(x,0)\le u_0(x)\leq B_{k_2}(x,0)
\end{equation*}
where
\begin{equation}\label{barenblett}
B_k(x,t)=\frac{2(n-2)(T-t)_+^{\frac{n}{n-2}}}{k+(T-t)_+^{\frac{2}{n-2}}|x|^2},
\qquad k>0,
\end{equation}
is the Barenblatt solution of \eqref{log-eqn} is studied by K.M.~Hui and 
S.~Kim in \cite{HK}. 

We will now assume that $n\ge 3$ and let $\beta>0$ be a fixed constant
for the rest of the paper. For any $\lambda>0$, let $\psi=\psi_{\lambda}$ 
be the radially symmetric solution of 
\begin{equation}\label{psi-eqn}
\begin{cases}
\La\log\psi+2\beta\psi+\beta x\cdot\nabla\psi=0,\quad \psi>0,\qquad 
\mbox{in }\R^n\\
\psi(0)=\lambda.
\end{cases}
\end{equation}
given by \cite{Hs4} and
\begin{equation}\label{self-similar-soln}
\phi=\phi_{\lambda}(x,t)=e^{-2\beta t}\psi_{\lambda}\left(e^{-\beta t}x\right).
\end{equation}
Whenever there is no ambiguity, we will drop the subscript $\lambda$ and write
$\psi$, $\phi$, instead of $\psi_{\lambda}$, $\phi_{\lambda}$. 
Then $\phi$ satisfies 
$$
\phi_{t}=\La\log\phi\quad\mbox{ in }\R^n\times (0,\infty).
$$
It was proved by S.Y.~Hsu in \cite{Hs4} that the radially symmetric 
solution $\psi$ of \eqref{psi-eqn} satisfies 
\begin{equation}\label{psi-decay-infty}
\lim_{r\to\infty}\frac{r^2\psi(r)}{\log r}=\frac{2(n-2)}{\beta}.
\end{equation}
A natural question to ask is that if the initial value $u_0$ has the same
decay rate at infinity as $\phi (x,0)=\psi(x)$ given by 
\eqref{self-similar-soln}, does the solution $u$ of \eqref{log-eqn} 
behaves like the function $\phi$ as $t\to\infty$. We answer this question 
in the affirmative in this paper. We prove that if the initial value 
$u_0$ satisfies
\begin{equation*}
u_0(x)\approx \frac{2(n-2)}{\beta}\frac{\log |x|}{|x|^2}\quad\mbox{ as }
|x|\to\infty
\end{equation*}
and $u$ is the global maximal solution of \eqref{log-eqn} with $n\ge 3$, then 
the rescaled function
\begin{equation}\label{rescale-soln}
\4{u}(x,t)=e^{2\beta t}u(e^{\beta t}x,t)
\end{equation}
converges uniformly on every compact subset of $\R^n$ to $\psi_{\lambda_0}$
as $t\to\infty$ for some constant $\lambda_0>0$. More precisely we prove 
the following main results of the paper.

\begin{thm}\label{main-thm}
Let $n\ge 3$ and $u_0$ satisfies
\begin{equation}\label{u0-trapped-between}
\psi_{\lambda_1}(x)\le u_0(x)\le\psi_{\lambda_2}(x)\quad\forall x\in\R^n
\end{equation}
and
\begin{equation}\label{u0-psi-L1}
u_0-\psi_{\lambda_0}\in L^1(\R^n)
\end{equation}
for some constants $\lambda_2>\lambda_1>0$ and $\lambda_0>0$. Suppose $u$ 
is the global maximal solution of \eqref{log-eqn} and $\4{u}$ is given by 
\eqref{rescale-soln}. Then $\4{u}$ converges uniformly on every compact 
subset of $\R^n$ and in $L^1(\R^n)$ to $\psi_{\lambda_0}$ as $t\to\infty$. 
Moreover
\begin{equation}\label{u-tilde-psi-L1-decay}
\|\4{u}(\cdot,t)-\psi_{\lambda_0}\|_{L^1(\R^n)}\le e^{-(n-2)\beta t}
\|u_0-\psi_{\lambda_0}\|_{L^1(\R^n)}\quad\forall t\ge 0.
\end{equation}
\end{thm}

\begin{thm}\label{main-thm2}
Let $n\ge 3$. Suppose $u_0$ satisfies
\begin{equation}\label{u0-w-cond1}
0\le u_0(x)\le\psi_{\lambda_1}(x)\quad\forall x\in\R^n
\end{equation}
and
\begin{equation}\label{u0-w-cond2}
\left|u_0(x)-\psi_{\lambda_{0}}(x)\right|\leq f(|x|)\in L^1(\R^n) 
\end{equation}
for some nonnegative radially symmetric function $f$ where $\psi_{\lambda_{0}}$, 
$\psi_{\lambda_1}$, are the radially symmetric solutions of \eqref{psi-eqn} 
with $\lambda=\lambda_0$, $\lambda_1$, respectively. 
Then the rescaled function $\4{u}(x,t)$ given by \eqref{rescale-soln} 
satisfies \eqref{u-tilde-psi-L1-decay} and converges uniformly on $\R^n$ 
and in $L^1(\R^n)$ to $\psi_{\lambda_0}$ as $t\to\infty$.
\end{thm}

Note that by Lemma \ref{psi-monotone-in-eta-lem} proved later that
$\psi_{\lambda}(x)$ is a monotone increasing function of $\lambda>0$. Hence 
\eqref{u0-trapped-between} is well-defined. Also by Lemma 
\ref{2psi-difference-not-L1-lem} proved later that the condition 
\eqref{u0-psi-L1} is necessary to guarantee convergence of the rescaled 
function $\4{u}$ as $t\to\infty$.

Unless stated otherwise we will assume that \eqref{u0-trapped-between} holds 
for the rest of the paper. Then by \eqref{psi-decay-infty},
\eqref{u0-trapped-between}, and the result of \cite{H3} there exists a 
unique global maximal solution $u$ of \eqref{log-eqn} for $n\ge 3$. 
Note that by direct computation $\tilde{u}$ given by \eqref{rescale-soln}
satisfies
\begin{equation}\label{u-tilde-eqn}
\tilde{u}_t=\La\log\tilde{u}+2\beta\tilde{u}+\beta x\cdot\nabla\tilde{u}
\quad\mbox{ in }\R^n\times (0,\infty).
\end{equation}
Then by \eqref{self-similar-soln}, \eqref{rescale-soln} and 
\eqref{u0-trapped-between}, 
\begin{align}
&\phi_{\lambda_1}(x,t)\le u(x,t)\le\phi_{\lambda_2}(x,t)\quad\forall x\in\R^n,t\ge 0
\notag\\
\Rightarrow\quad&\psi_{\lambda_1}(x)\le\4{u}(x,t)\le\psi_{\lambda_2}(x)
\qquad\forall x\in\R^n,t\ge 0.\label{u-tilde-u-l-bd}
\end{align}
The plan of the paper is as follows. In section 
\ref{section-property-selfsimilar-soln} we will recall and establish some
properties of the self-similar solution $\phi$. We will prove Theorem 
\ref{main-thm} and Theorem \ref{main-thm2} in section 3 and section 4 
respectively. 

We start with some definitions. We say that $u$ is a solution of 
\eqref{log-eqn} if $u>0$ in $\R^n\times (0,\infty)$ and $u$ 
satisfies 
\begin{equation*}
u_t=\Delta\log u
\end{equation*}
in the classical sense in $\R^n\times (0,\infty)$ with
$$
u(\cdot,t)\to u_0\quad\mbox{ in }L_{loc}^1(\R^n) \quad\mbox{ as }t\to 0.
$$
We say that $u$ is a maximal solution of \eqref{log-eqn} in 
$\R^n\times (0,\infty)$ if $u$ is a solution of \eqref{log-eqn} in 
$\R^2\times (0,T)$ and $u\ge v$ for any solution $v$ of 
\eqref{log-eqn} in $\R^n\times (0,T)$. For any $R>0$ and $x_0\in\R^n$, let 
$B_R(x_0)=\{x\in\R^N:|x-x_0|<R\}$ and $B_R=B_R(0)$.
Let $\omega_n$ be the surface area of the unit sphere $S^{n-1}$ in $\R^n$.
For any $a\in\R$, let $a_{\pm}=\max (\pm a,0)$.

\section{Properties of the self-similar solution}
\label{section-property-selfsimilar-soln}
\setcounter{equation}{0}
\setcounter{thm}{0}

In this section we will recall and establish some properties of the 
self-similar solution $\phi$. We first recall a result of \cite{Hs4}.

\begin{lemma}(cf. Lemma 1.1 and Theorem 1.3 of \cite{Hs4})
\label{existence-property-self-similar-soln}
Let $n\ge 2$, $\lambda>0$, $\alpha,\beta\in\R$, such that either $\alpha\ge 0$
or $\beta>0$. Then there exists a unique solution $v$ of
\begin{equation}
\label{ode-eqn}
\left(\frac{v'}{v}\right)'+\frac{n-1}{r}\cdot\frac{v'}{v}
+\alpha v +\beta rv'=0,\quad v>0,\qquad\mbox{ in }(0,\infty)
\end{equation}
which satisfies
\begin{equation}\label{psi-at-zero}
v(0)=\lambda \qquad \mbox{and} \qquad v'(0)=0.
\end{equation}
Moreover
\begin{equation}\label{eq-property-v-+-k-r-v-'-=-positive}
v+\frac{\beta}{\alpha}rv'>0 \quad \mbox{in }[0,\infty)\quad\mbox{ if }
\alpha\ne 0
\end{equation}
and 
\begin{equation*}
\left\{\begin{aligned}
&v'<0 \quad \mbox{ in }(0,\infty)\quad\mbox{ if }\alpha>0\\
&v'>0 \quad \mbox{ in }(0,\infty)\quad\mbox{ if }\alpha<0.
\end{aligned}\right.
\end{equation*}
\end{lemma}

\begin{lemma}\label{lemma-existence-and-uniqueness-plus-property-at-infty}
Let $n\ge 3$, $\lambda>0$, $\alpha=2\beta>0$, and let $\psi$ be the unique 
solution of \eqref{ode-eqn}, \eqref{psi-at-zero}, in 
$(0,\infty)$. Then $\psi$ satisfies
\begin{equation}\label{eq-asymptotic-of-r-2-w-at-infty}
\lim_{r\to\infty}r^2\left(\psi(r)+\frac{1}{2}r\psi'(r)\right)=\frac{n-2}{\beta}
\end{equation}
\end{lemma}

\begin{proof}
Let
\begin{equation}\label{eq-denote-r-4-v-by-w-1}
w(r)=r^4\left(\psi (r)+\frac{1}{2}r\psi '(r)\right).
\end{equation}
By direct computation $w(r)$ satisfies
\begin{equation*}
w_r+\left(\frac{n-6}{r}-\frac{\psi_r}{\psi}+\beta r\psi\right)w=(n-2)r^3\psi
\quad\forall r>0.
\end{equation*}
Hence
\begin{equation*}
\left(r^{n-6}f(r)w(r)\right)_r=(n-2)r^{n-3}f(r)\psi(r)\quad\forall r>0
\end{equation*}
where
\begin{equation*}
f(r)=\frac{\lambda}{\psi (r)}e^{\beta\int_0^r\rho\psi (\rho)\,d\rho}.
\end{equation*}
Integrating over $r\ge 1$,
\begin{equation}\label{eq-simplifying-of-w-over-r-after-integration-0}
\frac{w(r)}{r^2}=\frac{f(1)w(1)}{r^{n-4}f(r)}+\frac{(n-2)\int_{1}^{r}
\rho^{n-3}f(\rho)\psi(\rho)\,d\rho}{r^{n-4}f(r)}\quad\forall r>1.
\end{equation}
By \eqref{psi-decay-infty} there exists a constant 
$\rho_0\ge 1$ such that 
\begin{equation}\label{psi-u-l-bd}
\frac{4(n-2)\log\rho}{\beta\rho^2}>\psi(\rho)
>\frac{(n-2)\log\rho}{\beta\rho^2}, \qquad \forall 
\rho>\rho_0.
\end{equation}
Then 
\begin{align}
&f(\rho)\psi(\rho)\ge\lambda e^{(n-2)\int_{\rho_0}^{\rho}
\frac{\log s}{s}\,ds}\ge\lambda e^{\frac{n-2}{2}
\left[(\log\rho)^2-(\log\rho_0)^2\right]}\quad\forall \rho>\rho_0
\label{eq-lower-bound-of-f-times-v-0}\\
\Rightarrow\quad&\int_{1}^{r}\rho^{n-3}f(\rho)\psi(\rho)\,d\rho\to\infty
\quad\mbox{ as }r\to\infty\label{integral-f-to-infty}.
\end{align}
By \eqref{psi-u-l-bd} and \eqref{eq-lower-bound-of-f-times-v-0},
\begin{equation}\label{denumerator-to-infty}
r^{n-4}f(r)=\left(\frac{r^{n-4}}{\psi (r)}\right)f(r)\psi(r)
\geq \frac{\beta r^{n-2}f(r)\psi(r)}{4(n-2)\log r}\to\infty \qquad 
\mbox{as $r\to\infty$}.
\end{equation}
By \eqref{eq-simplifying-of-w-over-r-after-integration-0},  
\eqref{integral-f-to-infty},  \eqref{denumerator-to-infty},
and the  l'Hospital rule,
\begin{equation}\label{w-r2-limit-eqn}
\lim_{r\to\infty}\frac{w(r)}{r^2}
=\lim_{r\to\infty}\frac{(n-2)r^{n-3}f\psi}{(n-4)r^{n-5}f+r^{n-4}
\left(\beta r\psi-\frac{\psi_r}{\psi}\right)f}
=\lim_{r\to\infty}\frac{(n-2)r^2\psi}{(n-4)+r\left(\beta r\psi
-\frac{\psi_r}{\psi}\right)}.
\end{equation}
By (2.26) in \cite{Hs4}, we get
\begin{equation}\label{identity1}
\lim_{r\to\infty}\left(2+\frac{r\psi_r}{\psi}\right)=0.
\end{equation}
Since
\begin{equation}\label{identity2}
(n-4)+r\left(\beta r\psi-\frac{\psi_r}{\psi}\right)
=(n-2)+\beta r^2\psi-\left(2+\frac{r\psi_r}{\psi}\right)
\end{equation}
and by \eqref{psi-decay-infty} $r^2\psi(r)\to\infty$ as $r\to\infty$,
by \eqref{w-r2-limit-eqn}, \eqref{identity1}, and \eqref{identity2},
we have 
\begin{equation*}
\lim_{r\to\infty}\frac{w(r)}{r^2}
=\lim_{r\to\infty}\frac{(n-2)r^2\psi}{n-2+\beta r^2\psi}=\frac{(n-2)}{\beta}
\end{equation*}
and the lemma follows.
\end{proof}

\begin{lemma}\label{psi-monotone-in-eta-lem}
Let $n\beta>\alpha>0$, $\lambda_2>\lambda_1>0$, and let $v_{\lambda_1}$, 
$v_{\lambda_2}$ be the radially symmetric solutions of \eqref{ode-eqn}, 
\eqref{psi-at-zero}, in $(0,\infty)$ with $\lambda=\lambda_1$, 
$\lambda_2$, respectively. Then
\begin{equation*}
v_{\lambda_2}(r)>v_{\lambda_1}(r)>0, \qquad \forall r\geq 0.
\end{equation*} 
\end{lemma}
\begin{proof}
Let $\lambda>0$ and let $v=v_{\lambda}$ be the solution of \eqref{ode-eqn}, 
\eqref{psi-at-zero}, in $(0,\infty)$. Multiplying \eqref{ode-eqn} by $r^{n-1}$ 
and integrating, 
\begin{align}
r^{n-1}\frac{v_r(r)}{v(r)}=&-\alpha\int_0^r\rho^{n-1}v(\rho)
\,d\rho-\beta\int_0^r\rho^nv_{r}(\rho)\,d\rho\notag\\
=&-\beta r^nv(r)+(n\beta -\alpha)\int_0^r\rho^{n-1}v(\rho)\,d\rho 
\qquad\quad\forall r>0\notag\\
\Rightarrow\quad
v_{r}(r)=&-\beta rv^2(r)+\frac{(n\beta -\alpha)v(r)}{r^{n-1}}
\int_0^r\rho^{n-1}v(\rho)\,d\rho \quad \forall r>0.
\label{eq-for-psi-after-integration-1}
\end{align}
Since $\lambda_2>\lambda_1$, there exists $r_0>0$ such that $(0,r_0)$ is the 
maximal interval such that
\begin{equation*}
v_{\lambda_2}(r)>v_{\lambda_1}(r)>0, \qquad \forall 0\leq r<r_0.
\end{equation*} 
Suppose $r_0<\infty$. Then
\begin{equation}\label{eq-condition-of-psi-eta-1-eta-2-at-r-0-1}
v_{\lambda_2,r}(r_0)\leq v_{\lambda_1,r}(r_0), \quad v_{\lambda_2}(r_0)
=v_{\lambda_1}(r_0) \quad \mbox{and} \quad \int_0^{r_0}\rho^{n-1}
v_{\lambda_2}(\rho)\,d\rho>\int_0^{r_0}\rho^{n-1}v_{\lambda_1}(\rho)\,d\rho.
\end{equation}
Hence, by \eqref{eq-for-psi-after-integration-1} and 
\eqref{eq-condition-of-psi-eta-1-eta-2-at-r-0-1},
\begin{align}\label{eq-condition-of-psi-eta-1-eta-2-at-r-0-2}
v_{\lambda_2,r}(r_0)&=-\beta r_0v_{\lambda_2}^2(r_0)
+\frac{(n\beta -\alpha)v_{\lambda_2}(r_0)}{r_0^{n-1}}
\int_0^{r_0}\rho^{n-1}v_{\lambda_2}(\rho)\,d\rho\notag\\
&>-\beta r_0v_{\lambda_1}^2(r_0)+\frac{(n\beta -\alpha)
v_{\lambda_1}(r_0)}{r_0^{n-1}}\int_0^{r_0}\rho^{n-1}v_{\lambda_1}(\rho)\,d\rho
=v_{\lambda_1,r}(r_0).
\end{align}
By \eqref{eq-condition-of-psi-eta-1-eta-2-at-r-0-1} and 
\eqref{eq-condition-of-psi-eta-1-eta-2-at-r-0-2} contradiction arises. 
Hence $r_0=\infty$ and the lemma follows.
\end{proof}

\begin{lemma}\label{2psi-difference-not-L1-lem}
Let $n\ge 3$, $\beta>0$, and $\psi_{\lambda}$ be the radially symmetric solution 
of \eqref{psi-eqn} for any $\lambda>0$.  Then $\psi_{\lambda_2}-\psi_{\lambda_1}
\notin L^1(\R^{N})$ for any $\lambda_2>\lambda_1>0$.
\end{lemma}
\begin{proof}
Let $\alpha=2\beta$ and $\lambda_2>\lambda_1>0$. Since $\psi_1$ satisfies 
\eqref{ode-eqn} and \eqref{psi-at-zero} with $\lambda=1$, the function 
$\lambda\psi_1(\sqrt{\lambda}r)$ is a solution of \eqref{ode-eqn} and 
\eqref{psi-at-zero} for any $\lambda>0$. Since $\psi_{\lambda}$ also satisfies 
\eqref{ode-eqn}, \eqref{psi-at-zero}, by 
Lemma \ref{existence-property-self-similar-soln},
\begin{equation}\label{self-similar-form}
\psi_{\lambda}(x)=\psi_{\lambda}(|x|)=\lambda\psi_1(\sqrt{\lambda}|x|)
\quad\forall x\in\R^n.
\end{equation}
By \eqref{self-similar-form} and 
Lemma \ref{existence-property-self-similar-soln},
\begin{equation}\label{eq-difference-between-two-radially-symmetric-sol-49}
\psi_{\lambda_2}(x)-\psi_{\lambda_1}(x)=\int_{\lambda_1}^{\lambda_2}
\frac{\partial \psi_{\lambda}}{\partial\lambda}\,d\lambda
=\int_{\lambda_1}^{\lambda_2}
\left(\psi_1(\sqrt{\lambda}|x|)+\frac{\sqrt{\lambda}|x|}{2}\psi_1'
(\sqrt{\lambda}|x|)\right)\,d\lambda>0\quad\forall x\in\R^n.
\end{equation}
Hence
\begin{align}
\int_{\R^n}\left(\psi_{\lambda_2}-\psi_{\lambda_1}\right)\,dx
&=\int_{\R^n}\left[\int_{\lambda_1}^{\lambda_2}\left(\psi_1(\sqrt{\lambda}|x|)
+\frac{\sqrt{\lambda}|x|}{2}\psi_1'(\sqrt{\lambda}|x|)\right)
\,d\lambda\right]\,dx\notag\\
&=\int_{\lambda_1}^{\lambda_2}\left[\int_{\R^n}\left(\psi_1(\sqrt{\lambda}|x|)
+\frac{\sqrt{\lambda}|x|}{2}\psi_1'(\sqrt{\lambda}|x|)\right)\,dx\right]
\,d\lambda\notag\\
&=\omega_n
\int_0^{\infty}\rho^{n-1}\left(\psi_1(\rho)+\frac{\rho}{2}\psi_1'(\rho)
\right)\,d\rho\cdot\int_{\lambda_1}^{\lambda_2}\lambda^{-\frac{n}{2}}
\,d\lambda\notag\\
&=\frac{2\omega_n}{n-2}(\lambda_1^{1-\frac{n}{2}}-\lambda_2^{1-\frac{n}{2}})
\int_0^{\infty}\rho^{n-1}\left(\psi_1(\rho)+\frac{\rho}{2}\psi_1'(\rho)
\right)\,d\rho.\label{integral-psi1-psi2}
\end{align}
By Lemma \ref{lemma-existence-and-uniqueness-plus-property-at-infty}
there exist constants $C>0$ and $\rho_0>0$ such that
\begin{equation*}
\psi_1(\rho)+\frac{\rho}{2}\psi_1'(\rho)> \frac{C}{\rho^2},\qquad\forall
\rho>\rho_0.
\end{equation*}
Since $n\ge 3$, the right hand side of \eqref{integral-psi1-psi2}
is equal to infinity and the lemma follows.
\end{proof}

\section{Asymptotic Behavior}\label{section-Asymptotic-Behavior}
\setcounter{equation}{0}
\setcounter{thm}{0}

In this section we will use a modification of the technique of \cite{DS1}
and \cite{HK} to prove the asymptotic large time behaviour of the global 
maximal solution of \eqref{log-eqn}.  

\begin{lemma}\label{L1-t-L1-time0-compare-lem1}
Let $\lambda>0$ and $\phi=\phi_{\lambda}$ be given by \eqref{self-similar-soln}
where $\psi_{\lambda}=\psi$ is the radially symmetric solution of 
\eqref{psi-eqn}.
Suppose $u$, $v$, are solutions of \eqref{log-eqn} with inital values 
$u_0$, $v_0$, respectively which satisfy $u$, $v\ge\phi$ in $\R^n\times 
(0,\infty)$. Then for any $T>0$ there exist constants $R_0>0$ and 
$C>0$ depending on $T$ such that 
\begin{equation*}
(i)\quad \left(\int_{B_{R}(x)}(u-v)_+(y,t)\,dy\right)^{\frac{1}{2}}\leq 
\left(\int_{B_{2R}(x)}(u_0-v_0)_+(y)\,dy\right)^{\frac{1}{2}}
+C\left(\frac{R^{n-2}}{\log R}\right)^{\frac{1}{2}}
\end{equation*}
and 
\begin{equation*}
(ii)\quad \left(\int_{B_{R}(x)}|u-v|(y,t)\,dy\right)^{\frac{1}{2}}
\leq \left(\int_{B_{2R}(x)}|u_0-v_0|(y)\,dy\right)^{\frac{1}{2}}
+C\left(\frac{R^{n-2}}{\log R}\right)^{\frac{1}{2}}
\end{equation*}
holds for any $R\ge R_0+|x|$, $x\in\R^n$, $0\le t\le T$.
\end{lemma}
\begin{proof}
Let $T>0$. By \eqref{self-similar-soln} and \eqref{psi-decay-infty}
there exist constants $R_1\ge 1$, $C_1>0$ such that
\begin{align}
&\frac{r^2\psi(r)}{\log r}\ge C_1\quad\forall r\ge R_1\notag\\
\Rightarrow\quad&\left(\phi(y,t)\right)^{-1}
\le\frac{|y|^2}{C_1\log(e^{-\beta t}|y|)}
\le\frac{2|y|^2}{C_1\log|y|}\quad\forall |y|\ge R_0
:=e^{2\beta T}R_1, 0\le t\le T.\label{phi-lower-bd}
\end{align}
Then by an argument similar to the proof of Lemma 2.1 in \cite{HK} but with 
the $B_k$ and (2.4) there being replaced by $\phi$ and \eqref{phi-lower-bd}, 
(i) and (ii) of the lemma follows.
\end{proof}

\begin{lemma}\label{L1-bd-lem}
Let $\lambda>0$ and $\phi=\phi_{\lambda}$ be given by \eqref{self-similar-soln}
where $\psi_{\lambda}$ is the radially symmetric solution of \eqref{psi-eqn}.
Suppose $u$, $v$, are solutions of \eqref{log-eqn} with inital values 
$u_0$, $v_0$, respectively which satisfy $u$, $v\ge\phi$ in $\R^n\times 
(0,\infty)$. If $f=u_0-v_0\in L^{1}(\R^n)$, then $u(\cdot,t)-v(\cdot,t)\in 
L^1(\R^n)$ for all $t\ge 0$.
\end{lemma}
\begin{proof}
We will use a modification of the proof of Lemma 2.1 of \cite{DS1} and 
Lemma 2.2 of \cite{HK} to prove the lemma. Since the proof is similar to 
that of \cite{DS1} and \cite{HK}, we will only sketch the argument here. Let
\begin{equation*}
w(x,t)=\int_0^t\left|(\log u-\log v)\right|(x,s)\,ds.
\end{equation*}
Then by the Kato inequality \cite{Ka},
\begin{equation*}
\La\left|\log u-\log v\right|\geq \textbf{sign}(u-v)\La
\left(\log u-\log v\right).
\end{equation*}
Hence by \eqref{log-eqn}, 
\begin{equation*}
\frac{\partial}{\partial t}|u-v|\leq \La\left|\log u-\log v\right|
\end{equation*}
in the distribution sense in $\R^n\times (0,\infty)$. Integrating the above 
inequality in time, 
\begin{equation}\label{eq-relation-bet-w-and-f-0}
\La w\geq -|f| \qquad \mbox{on $\R^n$}.
\end{equation}
Let 
\begin{equation*}
Z(x)=\frac{1}{n(n-2)\omega_n}\int_{\R^n}\frac{|f(y)|}{|x-y|^{n-2}}\,dy
\end{equation*}
denote the Newtonian potential of $|f|$. Then by 
\eqref{eq-relation-bet-w-and-f-0},
\begin{equation}\label{eq-relation-bet-w-and-Z-1}
\La(w-Z)\geq 0
\end{equation}
in the sense of distributions in $\R^n$. Similar to the proof of Lemma 2.2 of 
\cite{HK} by \eqref{eq-relation-bet-w-and-Z-1} and an approximation 
argument the lemma would follow if we can show that 
\begin{equation}\label{w-average-limit-0}
\lim_{R\to\infty}\frac{1}{R^n}\int_{B_R(x)}w(y,t)\,dy=0\qquad\forall 
x\in\R^n,t>0.
\end{equation}
Since
\begin{equation*}
(\log u-\log v)_+=\left(\log\left(\frac{u}{v}\right)\right)_+
\le C\left(\frac{u}{v}-1\right)_+^{\frac{1}{2}}
\le C\phi^{-\frac{1}{2}}|u-v|^{\frac{1}{2}}
\end{equation*}
and similarly,
\begin{equation*}
(\log u-\log v)_-\le C\phi^{-\frac{1}{2}}|u-v|^{\frac{1}{2}},
\end{equation*}
we have
\begin{align}
\int_{B_R(x)}w(y,t)\,dy\le&C\int_0^t\int_{B_R(x)}\phi^{-\frac{1}{2}}
|u-v|^{\frac{1}{2}}\,dyds\notag\\
&\le C\int_0^t\left(\int_{B_R(x)}\phi(y,s)^{-1}\,dy\right)^{\frac{1}{2}}
\left(\int_{B_R(x)}|u-v|\,dy\right)^{\frac{1}{2}}\,ds.\label{w-integral-bd1}
\end{align}
Let $T>0$ and $R_0>1$ be as in the proof of 
Lemma \ref{L1-t-L1-time0-compare-lem1}. Then \eqref{phi-lower-bd} holds.
By \eqref{phi-lower-bd},
\begin{align}\label{1/phi-integral-bd}
\int_{B_R(x)}\phi(y,s)^{-1}\,dy=&\int_{B_R(x)\cap B_{R_0}}\phi(y,t)^{-1}\,dy
+\int_{B_R(x)\setminus B_{R_0}}\phi(y,t)^{-1}\,dy\notag\\
\le&C+C\int_{R_0\le |y|\le R+|x|}\frac{|y|^2}{\log |y|}\,dy\notag\\
\le&C\left(1+\frac{R^{n+2}}{R_0}\right)\quad\forall R>|x|+R_0,0\le s\le T.
\end{align}
By \eqref{w-integral-bd1}, \eqref{1/phi-integral-bd}, and 
Lemma \ref{L1-t-L1-time0-compare-lem1},
\begin{equation}\label{w-integral-bd2}
\int_{B_R(x)}w(y,t)\,dy\le C'\left(1+\frac{R^{n+2}}{R_0}\right)^{\frac{1}{2}}
\left(\|f\|_{L^{1}(\R^n)}^{\frac{1}{2}}
+\left(\frac{R^{n-2}}{\log R}\right)^{\frac{1}{2}}\right)
\quad\forall R>|x|+R_0,0\le t\le T
\end{equation}
for some constant $C'>0$ depending on $T$. Dividing both side of 
\eqref{w-integral-bd2} by $R^n$ and letting $R\to\infty$ we get
\eqref{w-average-limit-0} for any $0<t<T$. Since $T>0$ is arbitrary,
\eqref{w-average-limit-0} holds for all $t>0$.  
\end{proof}

By an argument similar to the proof of Corollary 2.2 of \cite{DS1} but
with Lemma \ref{L1-bd-lem} replacing Lemma 2.1 of \cite{DS1} in the proof
there we get the following result.

\begin{lemma}\label{L1-contraction-lem}
Let $\lambda>0$ and $\phi=\phi_{\lambda}$ be given by \eqref{self-similar-soln}
where $\psi_{\lambda}$ is the radially symmetric solution of \eqref{psi-eqn}.
Suppose $u$, $v$, are solutions of \eqref{log-eqn} with inital values 
$u_0$, $v_0$, respectively which satisfy $u$, $v\ge\phi$ in $\R^n\times 
(0,\infty)$. If $f=u_0-v_0\in L^{1}(\R^n)$, then 
$$
\|u(\cdot,t)-v(\cdot,t)\|_{L^1(\R^n)}\le\|u_0-v_0\|_{L^1(\R^n)}\quad\forall
t\ge 0.
$$
Hence if $\4{u}$, $\4{v}$, are the rescale functions of $u$, $v$, given 
by \eqref{rescale-soln} respectively, then 
\begin{equation}\label{u-v-tilde-compare}
\|\4{u}(\cdot,t)-\4{v}(\cdot,t)\|_{L^1(\R^n)}\le e^{-(n-2)\beta t}
\|u_0-v_0\|_{L^1(\R^n)}\quad\forall t\ge 0.
\end{equation}
\end{lemma}
We are now ready for the proof of Theorem \ref{main-thm}.

\noindent{\ni{\it Proof of Theorem \ref{main-thm}}:}
Let $\{t_i\}_{i=1}^{\infty}$ be a sequence of positive numbers such that 
$t_i\ge 1$ for all $i\in\Z^+$ and $t_i\to\infty$ as $i\to\infty$. By 
\eqref{u-tilde-u-l-bd} the equation \eqref{u-tilde-eqn} is uniformly 
parabolic on $B_R\times [0,\infty)$ for any $R>0$. By the 
Schauder estimates for parabolic equation \cite{LSU} the sequence 
$\4{u}(x,t_i)$ is equi-H\"older continuous in $C^2$ on every compact subset
of $\R^n$. Hence by the Ascoli Theorem and a diagonalization argument
the sequence $\4{u}(x,t_i)$ has a subsequence which we may assume without loss
of generality to be the sequence itself that converges uniformly on
every compact subset of $\R^n$ to some function $g$ as $i\to\infty$. 
By Lemma \ref{L1-contraction-lem} \eqref{u-v-tilde-compare} holds
with $\4{v}=\psi_{\lambda_0}$. Hence 
\begin{align*}
&\|\4{u}(\cdot,t_i)-\psi_{\lambda_0}\|_{L^1(\R^n)}\le e^{-(n-2)\beta t_i}
\|u_0-\psi_{\lambda_0}\|_{L^1(\R^n)}\quad\forall i\in\Z^+\\
\Rightarrow\quad&\|g-\psi_{\lambda_0}\|_{L^1(\R^n)}=0\quad\mbox{ as }i\to\infty\\
\Rightarrow\quad&g(x)=\psi_{\lambda_0}(x)\quad\forall x\in\R^n.
\end{align*}
Hence $u(x,t_i)$ converges uniformly on every compact subset of $\R^n$ 
to $\psi_{\lambda_0}$ as $i\to\infty$. Since the sequence $\{t_i\}_{i=1}^{\infty}$
is arbitrary, $u(x,t)$ converges uniformly on every compact subset of $\R^n$ 
to $\psi_{\lambda_0}$ as $t\to\infty$. By Lemma \ref{L1-contraction-lem} 
we get \eqref{u-tilde-psi-L1-decay} and the theorem follows.
\hfill$\square$\vspace{6pt}

\section{A more general result}
\setcounter{equation}{0}
\setcounter{thm}{0}

In this section we will prove Theorem \ref{main-thm2} and extend the 
convergence result of Theorem \ref{main-thm} to initial data not necessarily 
satisfying condition \eqref{u0-trapped-between}. We first start with a 
weaker convergence theorem.

\begin{thm}\label{weak-convergence-thm}
Let $n\ge 3$. Suppose $0\le u_0\in L^{\infty}(\R^n)$ satisfies 
\eqref{u0-psi-L1} where $\psi_{\lambda_{0}}$ is the radially symmetric 
solutions of \eqref{psi-eqn} with $\lambda=\lambda_0$. Suppose $u$ 
is the maximal solution of \eqref{log-eqn} in $\R^n\times(0,\infty)$
and $\4{u}(x,t)$ is given by \eqref{rescale-soln}. Then
\begin{equation}\label{eq-L-1-contraction-of-improvement-123}
\int_{\R^n}|u(\cdot,t)-\phi_{\lambda_0}(\cdot,t)|\,dx
\leq \|u_0-\psi_{\lambda_0}\|_{L^1(\R^N)}
\end{equation}
and \eqref{u-tilde-psi-L1-decay} holds. Hence $\4{u}$ converges to 
$\psi_{\lambda_0}$ in $L^1(\R^n)$ as $t\to\infty$.
\end{thm}
\begin{proof}
Since the proof is similar to the proof of Lemma 5.2 of \cite{HK}, we will 
only sketch the proof here. For any $0<\lambda<\lambda_0$, let $u_{\lambda}$ 
be the maximal global solution of \eqref{log-eqn} (cf. \cite{H3}) in 
$\R^n\times(0,\infty)$ with initial value
\begin{equation*}
u_{0,\lambda}(x)=\max (\psi_{\lambda}(x),u_0(x)).
\end{equation*} 
Then by the maximal principle 
\begin{equation}\label{u-lambda-l-bd}
u_{\lambda}\ge\max (\phi_{\lambda}(x,t),u(x,t))\quad\mbox{ in }\R^n\times 
(0,\infty)\quad\forall 0<\lambda<\lambda_0.
\end{equation}
and $u_{\lambda}$ decreases and converges to $u$ uniformly on every compact
subset of $\R^n\times (0,\infty)$ as $\lambda\searrow 0$. By 
\eqref{u-lambda-l-bd} and Lemma \ref{L1-contraction-lem},
\begin{equation}\label{u-lambda-phi-contraction}
\int_{\R^n}|u_{\lambda}(\cdot,t)-\phi_{\lambda_0}(\cdot,t)|\,dx
\leq \|u_0-\psi_{\lambda_0}\|_{L^1(\R^N)}\quad\forall 0<\lambda<\lambda_0.
\end{equation}
Letting $\lambda\searrow 0$ in \eqref{u-lambda-phi-contraction}
we get \eqref{eq-L-1-contraction-of-improvement-123}. 
\eqref{u-tilde-psi-L1-decay}  then follows directly from 
\eqref{eq-L-1-contraction-of-improvement-123} and the lemma follows.
\end{proof}

We next observe that by \eqref{psi-decay-infty} and an argument similar to 
the proof of Lemma 5.3 and Corollary 5.4 of \cite{HK} we have the following
results.

\begin{prop}[cf. Corollary 2.8 of \cite{H3}]
\label{lem-cf-Corollary-2-8-in-hui-1-in-improvement}
Let $n\ge 3$, $\lambda_0>0$, and $g(x)=\psi_{\lambda_0}(x,0)-h(x)$ for some 
radially symmetric function $0\le h\in L^{\infty}(\R^n)\cap L^1(\R^n)$ 
such that $g(x)\ge 0$ on $\R^n$. Then there exists a unique maximal global 
solution $u$ of \eqref{log-eqn} in $\R^n\times(0,\infty)$ with initial value 
$g$.
\end{prop}

\begin{cor}
\label{lem-cf-Corollary-2-8-in-hui-1-in-improvement}
Let $n\ge 3$ and let $\psi_{\lambda_0}(x)-h(x)\le u_0(x)\le
\psi_{\lambda_0}(x)$ for some radially symmetric function $h\in 
L^{\infty}(\R^n)\cap L^1(\R^n)$ satisfying $0\le h(x)\le \psi_{\lambda_0}(x)$ on 
$\R^n$. Then there exists a unique maximal solution $u$ of \eqref{log-eqn} in 
$\R^n\times(0,\infty)$ satisfying $0\le u(x,t)\le \phi_{\lambda_0}(x,t)$ in 
$\R^n\times (0,\infty)$ with initial value $u_0$.
\end{cor}

\begin{lemma}\label{tilde-u-bounded-above-below}
Let $n\ge 3$ and $0\le u_0\in L^{\infty}(\R^n)$ satisfy \eqref{u0-w-cond2} for 
some non-negative radially symmetric function $f$. Suppose $u$ is the 
maximal solution of \eqref{log-eqn} and $\4{u}$ is given by 
\eqref{rescale-soln}. Then there exist positive constants $C_1$, $C_2$, $C_3$ 
such that 
\begin{equation}\label{u-tilde-u-l-bd0}
C_1\frac{e^{-C_3\|f\|_{L^1(\R^n)}}\log |x|}{1+|x|^2}\le\4{u}(x,t) 
\leq C_2\frac{e^{C_3\|f\|_{L^1(\R^n)}}\log |x|}{1+|x|^2}\quad\forall |x|\ge 3,t\ge 2.
\end{equation}
\end{lemma}
\begin{proof}
We will use a modification of the proof of Proposition 6.2 of \cite{DS1}
to prove the lemma. We will first prove \eqref{u-tilde-u-l-bd0} 
under the assumption that $u_0(x)$ is radially symmetric in $r=|x|\ge 0$. 
Let $u_{\lambda}$, $0<\lambda<\lambda_0$, be as in the proof of 
Theorem \ref{weak-convergence-thm} and $t\ge 2$. Similar
to the proof of Lemma \ref{L1-bd-lem} the function
\begin{equation*}
w_{\lambda}(x)=\int_{t-1}^t|\log u_{\lambda}-\log\phi_{\lambda_0}|(x,\tau)
\,d\tau
\end{equation*}
satisfies
\begin{equation}\label{w-subharmonic}
\La(w_{\lambda}-Z_{\lambda})\geq 0 \qquad \mbox{in }\R^n
\end{equation}
and
\begin{equation}\label{w-lambda-average-limit}
\lim_{R\to\infty}\frac{1}{R^n}\int_{B_R(x)}w_{\lambda}(y)\,dy=0\quad\forall x\in\R^n
\end{equation}
where 
\begin{equation*}
Z_{\lambda}(x)=\int_{|x|}^{\infty}\frac{1}{\omega_n\rho^{n-1}}
\int_{|y|\leq\rho}|u_{\lambda}-\phi_{\lambda_0}|(y,t-1)\,dy\,d\rho
\quad\forall x\in\R^n, t\ge 2,
\end{equation*}
is the Newtonian potential of $|u_{\lambda}-\phi_{\lambda_0}|(\cdot,t-1)$.
Then by \eqref{w-subharmonic}, \eqref{w-lambda-average-limit}, and the mean 
value property for subharmonic function,
\begin{equation*}
w_{\lambda}(x)\leq Z_{\lambda}(x) \qquad \mbox{in }\R^n\quad\forall t\ge 2.
\end{equation*}
Hence 
\begin{equation}\label{w-lambda-decay}
w_{\lambda}(x)\le C_3\frac{\|(u_{\lambda}-\phi_{\lambda_0})(\cdot,t-1)
\|_{L^1(\R^N)}}{|x|^{n-2}}\quad\forall |x|\ge 1, t\ge 2,
\end{equation}
for some constant $C_3>0$. By \eqref{u0-w-cond2}, \eqref{w-lambda-decay},
and Theorem \ref{weak-convergence-thm},
\begin{equation}\label{log-u-tilde-integral-ineqn}
\int_{t-1}^t\log \phi_{\lambda_0}(x,\tau)\,d\tau
-C_3\frac{\|f\|_{L^1(\R^N)}}{|x|^{n-2}}
\le\int_{t-1}^t\log u_{\lambda}(x,\tau)\,d\tau
\le\int_{t-1}^t\log\phi_{\lambda_0}(x,\tau)\,d\tau
+C_3\frac{\|f\|_{L^1(\R^N)}}{|x|^{n-2}}
\end{equation}
holds for any $|x|\ge 1$ and $t\ge 2$. By \eqref{self-similar-soln} and 
\eqref{eq-property-v-+-k-r-v-'-=-positive}, 
\begin{equation*}
\phi_{_{\lambda_0,t}}(r,t)=-2\beta e^{-2\beta t}\left(\psi_{\lambda_0}(\rho)
+\frac{1}{2}\rho\psi_{\lambda_0}'(\rho)\right)\leq 0\quad\forall r\ge 0,t>0,
\rho=e^{-\beta t}r.
\end{equation*} 
Hence
\begin{equation}\label{log-u-tilde-u-l-bd}
\log\phi_{\lambda_0}(x,t)\le\int_{t-1}^t\log\phi_{\lambda_0}(x,\tau)\,d\tau
\le\log\phi_{\lambda_0}(x,t-1).
\end{equation}
Since by Lemma \ref{existence-property-self-similar-soln} 
$\psi_{\lambda_0}'(r)<0$ for all $r>0$, we have
\begin{equation}\label{phi-t-(t-1)-compare}
\phi_{\lambda_0}(x,t-1)=e^{-2\beta (t-1)}\psi_{\lambda_0}(e^{-\beta (t-1)}x)
\le e^{2\beta}e^{-2\beta t}\psi_{\lambda_0}(e^{-\beta t}x)
=e^{2\beta}\phi_{\lambda_0}(x,t).
\end{equation} 
By \eqref{log-u-tilde-integral-ineqn}, \eqref{log-u-tilde-u-l-bd}, and 
\eqref{phi-t-(t-1)-compare},
\begin{equation}\label{phi-u-lambda-ineqn}
\log\left(\frac{\phi_{\lambda_0}(x,t-1)}{C_4}\right)
\le\int_{t-1}^t\log u_{\lambda}(x,\tau)\,d\tau 
\le\log \left(C_4\phi_{\lambda_0}(x,t)\right) \qquad \forall |x|\ge 1,
t\ge 2,
\end{equation}
where $C_4=e^{2\beta+C_3\|f\|_{L^1}}$. Since $u_{\lambda_0}$ satisfies the 
Aronson-Benilan inequality (cf. \cite{H3}), 
$$
u_t\le\frac{u}{t}\qquad\qquad\mbox{ in }\R^n\times (0,\infty), 
$$
we have
\begin{align}\label{u-lambda0-different-time-compare}
&\frac{\tau}{t}u_{\lambda}(x,t)\le u_{\lambda}(x,\tau)
\le\frac{\tau}{t-1}u_{\lambda}(x,t-1)
\qquad\qquad\qquad\qquad\qquad\forall x\in\R^n,
t-1\le\tau\le t, t\ge 2\notag\\
\Rightarrow\quad&\log\left(\frac{t-1}{t}\,u_{\lambda}(x,t)\right)
\le\int_{t-1}^t\log u_{\lambda}(x,\tau)\,d\tau \le\log\left(\frac{t}{t-1}
\,u_{\lambda}(x,t-1)\right)\quad\forall x\in\R^n,t\ge 2\notag\\
\Rightarrow\quad&\log\left(\frac{u_{\lambda}(x,t)}{2}\right)
\le\int_{t-1}^t\log u_{\lambda}(x,\tau)\,d\tau
\le\log(2u_{\lambda}(x,t-1))\qquad\qquad\qquad\forall x\in\R^n,t\ge 2.
\end{align}
By \eqref{phi-u-lambda-ineqn} and \eqref{u-lambda0-different-time-compare},
\begin{equation}\label{u-lambda-phi-ineqn}
\left\{\begin{aligned}
&\frac{u_{\lambda}(x,t)}{2}\le C_4\phi_{\lambda_0}(x,t)
\qquad\quad\,\,\,\,\,\forall |x|\ge 1,t\ge 2\\
&\frac{\phi_{\lambda_0}(x,t-1)}{C_4}\le 2u_{\lambda}(x,t-1)
\quad\forall |x|\ge 1,t\ge 2.
\end{aligned}\right.
\end{equation}
Letting $\lambda\to 0$ in \eqref{u-lambda-phi-ineqn},
\begin{align}\label{u-tilde-u-l-bd3}
&\frac{1}{2C_4}\phi_{\lambda_0}(x,t)\le u(x,t)
\le 2C_4\phi_{\lambda_0}(x,t)\quad\forall\,|x|\ge 1, t\ge 2\notag\\
\Rightarrow\quad&\frac{1}{2C_4}\psi_{\lambda_0}(y)\le\4{u}(y,t) 
\le 2C_4\psi_{\lambda_0}(y) \quad\forall\,|y|\ge 1, t\ge 2.
\end{align}
By \eqref{u-tilde-u-l-bd3} and \eqref{psi-decay-infty}, we get 
\eqref{u-tilde-u-l-bd0} for some constants $C_1>0$, $C_2>0$.

When $u_0(x)$ is not radially symmetric and satisfies \eqref{u0-w-cond2}, 
by the above result for the radially symmetric initial data case and an 
argument similar to the last step of the proof of Proposition 6.2 of 
\cite{DS1} on p.118 of \cite{DS1} we get \eqref{u-tilde-u-l-bd0} for some 
constants $C_1>0$, $C_2>0$, and the lemma follows.
\end{proof}

\begin{cor}\label{tilde-u-lower-bd-cor}
Let $n$, $u_0$, $u$, $\4{u}$, be as in 
Lemma \ref{tilde-u-bounded-above-below}. Then there exists a constant
$C_4>0$ such that
\begin{equation}\label{tilde-u-lower-upper-bd}
\4{u}(x,t)\ge C_4\frac{\max(1,\log |x|)}{1+|x|^2}
\quad\forall x\in\R^n,t\ge 2.
\end{equation}
\end{cor}
\begin{proof}
By Lemma \ref{tilde-u-bounded-above-below} there exist constants $C_1>0$, 
$C_2>0$, $C_3>0$, such that \eqref{u-tilde-u-l-bd0} holds. Let
$$
C_4'=\min\left((C_1/10)(\log 3)e^{-C_3\|f\|_{L^1(\R^n)}},\min_{|x|\le 3}\4{u}(x,2)
\right).
$$
Since $\4{u}$ satisfies \eqref{u-tilde-eqn}, by applying the maximal principle
to $\4{u}$ in $B_3\times (2,\infty)$ we get
\begin{equation}\label{u-tilde-l-bd1}
\4{u}(x,t)\ge C_4'\quad\forall |x|\le 3,t\ge 2.
\end{equation}
By \eqref{u-tilde-u-l-bd0} and \eqref{u-tilde-l-bd1} there exists a constant
$C_4>0$ such that \eqref{tilde-u-lower-upper-bd} holds.
\end{proof}

We are now ready for the proof of Theorem \ref{main-thm2}.

\noindent{\it Proof of Theorem \ref{main-thm2}}: 
By \eqref{u0-w-cond1} and the maximal principle,
\begin{align}
&0\le u(x,t)\le\phi_{\lambda_1}(x,t)\quad\forall x\in\R^n,t>0\notag\\
\Rightarrow\quad&0\le\4{u}(x,t)\le\psi_{\lambda_1}(x)\quad\forall x\in\R^n,t>0.
\label{u-tilde-bd-1}
\end{align}
By Corollary \ref{tilde-u-lower-bd-cor} and \eqref{u-tilde-bd-1} for any
$R>0$, there exist constants $C_5>0$, $C_6>0$, such that
\begin{equation}\label{u-tilde-uniformly-bd}
C_5\le\4{u}(x,t)\le C_6\quad\forall |x|\le R,t\ge 2.
\end{equation}
By \eqref{u-tilde-uniformly-bd} the equation \eqref{u-tilde-eqn} is uniformly 
parabolic on $B_R\times [2,\infty)$ for any $R>0$. Let $\{t_i\}_{i=1}^{\infty}$ 
be a sequence such that $t_i\ge 3$ for all $i\in\Z^+$ and $t_i\to\infty$ 
as $i\to\infty$. By the Schauder estimates for parabolic equation \cite{LSU} 
the sequence 
$\4{u}(x,t_i)$ is equi-H\"older continuous in $C^2$ on every compact subsets 
of $\R^n$. Then by the Arzela-Ascoli theorem and a diagonalization argument
the sequence $\{\4{u}(x,t_i)\}_{i=1}^{\infty}$ has a convergent 
subsequence which we may assume without loss of generality to be the 
sequence itself that converges uniformly in $C^2$ on every compact subsets 
of $\R^n$ to some $C^2$ function $w$ of as $i\to\infty$. On the other hand by 
Theorem \ref{weak-convergence-thm} $\4{u}(x,t)$ satisfies 
\eqref{u-tilde-psi-L1-decay}. Hence $\4{u}(x,t)$ converges to $\psi_{\lambda_0}$
in $L^1(\R^n)$ as $t\to\infty$ and $w=\psi_{\lambda_0}$. Thus $\4{u}(x,t_i)$
converges uniformly in $C^2$ on every compact subsets 
of $\R^n$ to $\psi_{\lambda_0}$ as $i\to\infty$. Since the 
sequence is arbitrary, $\4{u}(x,t)$ converges uniformly in $C^2$ on every 
compact subsets of $\R^n$ to $\psi_{\lambda_0}$ as $t\to\infty$ and the 
theorem follows.
\hfill$\square$

\vspace{6pt}


\begin{thebibliography}{99}

\bibitem[A]{A} D.G.~Aronson, {\em The porous medium equation, CIME Lectures}, 
in Some problems in Nonlinear Diffusion, Lecture Notes in Mathematics 1224, 
Springer-Verlag, New York, 1986.

\bibitem[DK]{DK} P.~Daskalopoulos and C.E.~Kenig, {\em Degenerate 
Diffusions-Initial Value Problems and Local Regularity Theory}, EMS Tracts 
in Mathematics 1, European Mathematical Society, 2007.  

\bibitem[DP1]{DP1} P.~Daskalopoulos and M.A.~del Pino, {\em On a singular 
diffusion equation}, Comm. Anal. Geom. 3 (1995), no. 3, 523--542.

\bibitem[DP2]{DP2} P.~Daskalopoulos and M.A.~del Pino, {\em On the Cauchy 
problem for $u_t=\Delta\log u$ in higher dimensions}, Math. Ann. 313
(1999), 189--206.

\bibitem[DP3]{DP3} P.~Daskalopoulos and M.A.~del Pino, {\em Type II 
collapsing of maximal solutions to the Ricci flow in $\R^2$}, Ann. Inst. 
H. Poincar\'e Anal. Non Linaire 24 (2007), 851--874. 

\bibitem[DS1]{DS1} P. Daskalopoulos, N. Sesum, {\em On the extinction profile 
of solutios to fast diffusion} J. Reine Angew. Math. 622 (2008), 95-119

\bibitem[DS2]{DS2} P.~Daskalopoulos and N.~Sesum, {\em Type II extinction 
profile of maximal solutions to the Ricci flow equation}, 
J. Geom. Anal. 20 (2010), no. 3, 565--591.

\bibitem[ERV]{ERV} J.R.~Esteban, A.~Rodriguez and J.L.~Vazquez, {\em The 
maximal solution of the logarithmic fast diffusion equation in two space 
dimensions}, Advances in Diff. Eq. 2 (1997), no. 6, 867--894.

\bibitem[Hs1]{Hs1} S.Y.~Hsu, {\em Global existence and uniqueness of solutions 
of the Ricci flow equation}, Differential Integral Equations 
14 (2001), no. 3, 305--320.

\bibitem[Hs2]{Hs2} S.Y.~Hsu, {\em Large time behaviour of solutions of the 
Ricci flow equation on $R^2$}, Pacific J. Math. 197 (2001), no. 1, 25--41.

\bibitem[Hs3]{Hs3} S.Y.~Hsu, {\em Asymptotic profile of solutions of a singular 
diffusion equation as $t\to\infty$}, Nonlinear Analysis TMA, 48 (2002), no. 6,
781--790.

\bibitem[Hs4]{Hs4} S.Y.~Hsu, {\em Classification of radially symmetric 
self-similar solutions of $u_t=\La\log u$ in higher dimensions}, 
Differential and Integral Equations, 18 (2005), no. 10, 1175-1192.

\bibitem[H1]{H1} K.M.~Hui, {\em Existence of solutions of the equation
$u_t=\Delta$log $u$}, Nonlinear Analysis TMA 37 (1999), no. 7, 875--914.

\bibitem[H2]{H2} K.M.~Hui, {\em Singular limit of solutions of the 
equation $u_t=\Delta (\frac{u^m}{m})$ as $m\to 0$}, Pacific J. Math. 187
(1999), no. 2, 297--316.

\bibitem[H3]{H3} K.M.~Hui, {\em On Some Dirichlet and Cauchy Problems for 
a Singular Diffusion Equation}, Differential Integral Equations 15 (2002), 
no. 7, 769-804.

\bibitem[H4]{H4} K.M.~Hui, {\em Collapsing behaviour of a singular diffusion 
equation}, to appear in Discrete Contin. Dynamical Systems-Series A.

\bibitem[HK]{HK} K.M.~Hui and S.~Kim, {\em Extinction profile of the 
logarithmic diffusion equation},
http://arxiv.org/abs \linebreak /1012.1915v2.

\bibitem[K]{K} T.G.~Kurtz, {\em Convergence of sequences of semigroups
of nonlinear operators with an application to gas kinetics}, Trans. Amer. 
Math. Soc. 186 (1973), 259--272.

\bibitem[Ka]{Ka} T. Kato, {\em Schr\"odinger operators with singular 
potentials}, Israel J. Math. 13 (1973), 135-148.

\bibitem[Ki]{Ki} J.R.~King, {\em Extremely high concentration dopant diffusion
in silicon}, IMA J. Appl. Math. 40 (1998), 163--181.

\bibitem[LSU]{LSU} O.A.~Ladyzenskaya, V.A.~Solonnikov and N.N.~Uraltceva, 
{\em Linear and Quasilinear Equations of Parabolic Type}, Transl. Math. 
Mono. vol. 23, Amer. Math. Soc., Providence, R.I., USA, 1968.

\bibitem[LT]{LT} P.L.~Lions and G.~Toscani, {\em Diffusive limit for
finite velocity Boltzmann kinetic models}, Revista Matematica Iberoamericana
13 (1997), no. 3, 473--513.

\bibitem[P]{P} L.A.~Peletier, {\em The porous medium equation} in
Applications of Nonlinear Analysis in the Physical Sciences,
H.~Amann, N.~Bazley, K.~Kirchgassner editors, Pitman, Boston, 1981.

\bibitem[W1]{W1} L.F.~Wu, {\em A new result for the porous medium equation 
derived from the Ricci flow}, Bull. Amer. Math. Soc. 28 (1993), 90--94.

\bibitem[W2]{W2} L.F.~Wu, {\em The Ricci flow on $R^2$}, Comm. Anal. Geom. 
1 (1993), 439--472.

\end{thebibliography}
\end{document}